\documentclass{amsart}

\usepackage[utf8]{inputenc}
\usepackage{amsmath,amsthm,amssymb}
\usepackage{indentfirst}
\usepackage{url}
\usepackage{tikz-cd}
\usepackage[colorlinks=true]{hyperref}
\usepackage{colonequals}
\usepackage{graphicx}
\usepackage{enumerate}
\usepackage{soul}
\usepackage{stmaryrd}
\usepackage{mathrsfs,mathtools}
\usepackage{tgpagella}
\usepackage{bbm}
\usepackage[alphabetic]{amsrefs}

\usepackage[margin=1in]{geometry}

\newtheorem{thm}[subsection]{Theorem}
\newtheorem{lem}[subsection]{Lemma}

\newtheorem{prop}[subsection]{Proposition}

\newtheorem{cor}[subsection]{Corollary}
{
\theoremstyle{definition}
\newtheorem{remx}[subsection]{Remark}

\newtheorem{defnx}[subsection]{Definition}
\newtheorem{examplex}[subsection]{Example}
}
\newenvironment{rem}
{\pushQED{\qed}\remx}
{\popQED\endremx}

\newenvironment{defn}
{\pushQED{\qed}\defnx}
{\popQED\enddefnx}
\newenvironment{example}
{\pushQED{\qed}\examplex}
{\popQED\endexamplex}

\newcommand{\mathblockyblocky}[1]{\mathbb{#1}}

\newcommand{\B}{\bigg}
\newcommand{\conv}{\textup{conv}}
\newcommand{\comp}{\textup{comp}}
\newcommand{\un}{\textup{un}}

\newcommand{\NN}{\mathblockyblocky N}
\newcommand{\ZZ}{\mathblockyblocky Z}

\newcommand{\RR}{\mathblockyblocky R}

\begin{document}
\title{Faces of root polytopes}
\author{Linus Setiabrata}
\address{Department of Mathematics, University of Chicago}
\email{linus@math.chicago.edu}
\begin{abstract}
For every directed acyclic graph $G$, we characterize the faces of the root polytope $\tilde Q_G = \conv\{\mathbf 0, \mathbf e_i - \mathbf e_j\colon (i,j) \in E(G)\}$ combinatorially. Our results specialize to state of the art results in a straightforward way.
\end{abstract}
\maketitle
\vspace{-0.5em}
\section{Introduction}
Let $A_n^+ = \{\mathbf e_i - \mathbf e_j\colon 1 \leq i < j \leq n+1\}\subset\RR^{n+1}$ denote the positive roots of type $A_n$. Subsets of $A_n^+$ can be encoded using a directed acyclic graph $G$ on $n+1$ vertices with edges $(i,j) \in E(G)$ oriented so that $i < j$. Given such a graph $G$, one can consider the \emph{root polytopes}
\[
Q_G \overset{\rm def}= \conv\{\mathbf e_i - \mathbf e_j\colon (i,j) \in E(G)\} \subset \RR^{n+1}
\]
and
\[
\tilde Q_G \overset{\rm def}= \conv\{\mathbf 0, \mathbf e_i - \mathbf e_j\colon (i,j) \in E(G)\} \subset \RR^{n+1}.
\]
The purpose of this paper is to completely characterize the faces of the root polytope $\tilde Q_G$ for every $G$. This is accomplished in Theorems~\ref{thm:tilde-faces} and~\ref{thm:non-tilde-faces}.\\

Root polytopes were first studied systematically in~\cite{postnikov2009}, where it was shown that the simplices in a triangulation of a root polytope count lattice points of a generalized permutahedron. The class of root polytopes also includes products of simplices, the triangulations of which are known to have very rich combinatorics (see e.g.\ \cite{hrs2000,santos2005,gnp2018}). Triangulations and subdivision algebras of root polytopes were studied in~\cite{meszaros2011,meszaros2016}, and have been used to solve a variety of other combinatorial problems, e.g.\ in~\cite{em2016,em2018}.

Much attention has been devoted to studying the face structure of the convex hull of the entire type $A_n$ root system, and more generally to that of other root systems $\Phi$. The faces of the polytope $\mathcal P_{A_n} = \conv\{\mathbf e_i - \mathbf e_j\colon i,j \in [n+1]\}$ were characterized combinatorially already in~\cite{cho1999}; computing the $f$-vector of $\mathcal P_{A_n}$ is an easy corollary of the characterization. The $f$-vectors of pulling triangulations of the boundary of $\mathcal P_{A_n}$ were computed in~\cite{hetyei2009}, and the $f$-vectors of unimodular triangulations of the boundary of $\mathcal P_\Phi = \conv\{\mathbf v\colon\mathbf v \in \Phi\}$, $\Phi = A_n, C_n, D_n$, were given in~\cite{abhps2011}. The orbit classes (under an action of the Weyl group) of the faces of $\mathcal P_\Phi$ were algebraically characterized in~\cite{cm2015}.

In contrast, to our knowledge the faces of convex hulls of (subsets of) \emph{positive} roots have been studied only for $\Phi^+ = A_n^+$. Gelfand, Graev, and Postnikov studied faces of $\tilde Q_{K_n}$ not containing the origin in~\cite[Prop.\ 8.1]{ggp1997}, but their result contains a mistake. Cho salvaged this result for facets of $\tilde Q_{K_n}$ in~\cite[Prop.\ 13]{cho1999}. Postnikov generalized Cho's result to facets of $\tilde Q_G$ for \emph{transitively closed} graphs $G$ (Definition~\ref{defn:transitively-closed}) in~\cite[Prop.\ 13.3]{postnikov2009}. To our knowledge, Postnikov's characterization~\cite[Prop.\ 13.3]{postnikov2009} has been the state of the art in this direction. Our results specialize to those of Postnikov straightforwardly (spelled out in Corollary~\ref{cor:transitively-closed-facets}), and correct the mistake in~\cite[Prop.\ 8.1]{ggp1997} in full generality (Corollary~\ref{cor:non-tilde-faces-K_n}; see also Remark~\ref{rem:ggp-comparison}).

When $G$ is an \emph{alternating} graph (Definition~\ref{defn:alternating}), the faces of the affine cone generated by $\{\mathbf e_i - \mathbf e_j\colon (i,j) \in E(G)\}$ has algebrogeometric significance: it is related to the deformation theory of a certain toric variety associated to $G$. The faces of this cone, i.e.\ the faces of $\tilde Q_G$ containing the origin, were combinatorially characterized in the recent paper~\cite[Thm.\ 3.17]{portakal2019}, building on the work in~\cite{vv2006}. We highlight and reprove their characterization in Corollaries~\ref{cor:portakal1} and~\ref{cor:portakal2}.

The faces of $\tilde Q_G$ are again root polytopes, i.e.\ equal to $\tilde Q_H\subseteq \tilde Q_G$ or $Q_H\subset \tilde Q_G$ for certain subgraphs $H\subseteq G$ (Proposition~\ref{prop:faces-of-root-are-root}). We characterize the subgraphs $H$ for which $\tilde Q_H\subseteq \tilde Q_G$ is a face in Theorem~\ref{thm:tilde-faces}, and separately characterize the subgraphs $H$ for which $Q_H\subset \tilde Q_G$ is a face in Theorem~\ref{thm:non-tilde-faces}. For $G = K_n$, the characterizations of Theorem~\ref{thm:tilde-faces} and~\ref{thm:non-tilde-faces} are particularly nice, and are highlighted in Corollary~\ref{cor:tilde-faces-K_n} and Corollary~\ref{cor:non-tilde-faces-K_n} respectively.

\section{Background}
\textbf{Conventions.} Unless stated otherwise, $G$ will denote a directed acyclic graph with $V(G) = [n]$. Without loss of generality, we may assume its edges $e = (i,j) \in E(G)$ are directed so that $i < j$. (The adjective \emph{acyclic} will only describe directed graphs, and means that there is no \emph{directed} cycle.) We use the notation $H\subseteq G$ to denote a subgraph $H$ of $G$ with $V(H) = V(G)$ and $E(H) \subseteq E(G)$. We also use the notation $G^\un$ to denote the underlying undirected graph of $G$. We reserve boldface mathematical notation to denote vectors; in particular $\mathbf e_i$ is the $i$-th basis vector of $\RR^n$. 

\textbf{Root polytopes.} In~\cite[Sec.\ 12]{postnikov2009}, Postnikov defined the \textbf{root polytopes} 
\[
Q_G \overset{\rm def}= \conv\{\mathbf e_i - \mathbf e_j\colon (i,j) \in E(G)\} \subset \RR^n
\]
and
\[
\tilde Q_G \overset{\rm def}= \conv\{\mathbf 0, \mathbf e_i - \mathbf e_j\colon (i,j) \in E(G)\} \subset \RR^n.
\]
It is well known that faces of root polytopes are again root polytopes:
\begin{prop}
\label{prop:faces-of-root-are-root}
For every subgraph $H\subseteq G$, the root polytope $Q_H$ is a subpolytope of $\tilde Q_H$, which in turn is a subpolytope of $\tilde Q_G$. Every subpolytope (in particular, every face) of $\tilde Q_G$ is the root polytope $Q_H$ or the root polytope $\tilde Q_H$ for some $H\subseteq G$.
\end{prop}
\begin{proof}
The inclusion of edge sets $E(H) \subseteq E(G)$ implies the inclusion of polytopes $Q_H\subset\tilde Q_H\subseteq\tilde Q_G$.

Conversely, every subpolytope $P$ of $\tilde Q_G$ is the convex hull of the vertices of $\tilde Q_G$ which live in $P$ (see e.g.~\cite[Prop.\ 2.3]{ziegler2007}). The non-origin vertices correspond to edges of $G$, so the collection of such vertices forms a subgraph $H$ of $G$. If $P$ contains (resp.\ doesn't contain) the origin, then $P = \tilde Q_H$ (resp.\ $P = Q_H$).
\end{proof}

\begin{defn}
\label{defn:alternating}
A graph $G$ is \textbf{alternating} if there is no vertex $j \in [n] = V(G)$ so that $(i,j), (j,k) \in E(G)$.
\end{defn}

We remark that alternating graphs are nothing more than (appropriately oriented) bipartite graphs:

\begin{lem}
\label{lem:alternating-is-bipartite}
Let $G$ be an alternating graph and suppose $G^\un$ is connected. Then there is a partition of $V(G) = L\sqcup R$ into two parts so that every edge $(i,j) \in E(G)$ connects a vertex $i \in L$ to a vertex $j \in R$.
\end{lem}

\begin{proof}
If $G$ has no edges, the lemma is vacuous. Otherwise, we may set
\begin{align*}
L&\overset{\rm def}=\{v \in V(G)\colon \textup{every edge of $G$ incident to $v$ has $v$ as its source}\},\\
R&\overset{\rm def}=\{v \in V(G)\colon \textup{every edge of $G$ incident to $v$ has $v$ as its sink}\}.
\end{align*}
Every vertex of the alternating graph $G$ has an edge incident to it, so $L$ and $R$ are disjoint. If a vertex $j \in [n]$ is not in $L$, then there is an edge $(i,j) \in E(G)$ with $j$ as its sink; similarly if $j$ is not in $R$, then there is an edge $(j,k) \in E(G)$ with $j$ as its source. Since $G$ is alternating, these cannot simultaneously happen, so $j\in L\sqcup R$. We conclude $L\sqcup R = [n]$.

From the definitions of $L$ and $R$, we see that every edge of $G$ connects a vertex in $L$ to a vertex in $R$.
\end{proof}

The following result can be derived from~\cite{postnikov2009}. Here we include a full proof for completeness.

\begin{prop}[{cf.\ \cite[Lem.\ 13.2, Lem.\ 12.5]{postnikov2009}}]
\label{prop:root-polytope-dimension-general}
Suppose $G^\un$ has $r$ connected components. Then $\tilde Q_G$ is $(n-r)$-dimensional. If $G^\un$ has $r$ connected components and $G$ is alternating, then $Q_G$ is $(n-r-1)$-dimensional.
\end{prop}

\begin{proof}
Take a spanning forest $T^\un\subseteq G^\un$ and let $T\subseteq G$ be its overlying directed graph. The $n-r+1$ vertices of $\tilde Q_T\subseteq \tilde Q_G$ are affinely independent and hence form an $(n-r)$-dimensional simplex. On the other hand, $\tilde Q_G$ is contained in the $(n-r)$-dimensional subspace
\[
W = \B\{\mathbf x \in \RR^n \colon \sum_{i \in G_j^\un} x_i = 0 \textup{ for all connected components $G_j^\un$ of $G^\un$}\B\}\subset\RR^n.
\]
It follows that $\tilde Q_G$ is $(n-r)$-dimensional.

Suppose now that $G^\un$ has $r$ connected components and $G$ is alternating. In this case, there is a subset $L\subseteq [n] = V(G)$ so that every edge $e \in E(G)$ has source in $L$ and target not in $L$ (the set $L$ can be thought of as ``source vertices'' of the graph $G$).

As before, take a spanning forest $T^\un\subseteq G^\un$ and let $T\subseteq G$ be its overlying directed graph. The $n-r$ vertices of $Q_T\subseteq Q_G$ are affinely independent and hence form an $(n-r-1)$-dimensional simplex. On the other hand, $\tilde Q_G$ is contained in the $(n-r)$-dimensional subspace $W$ and also in the subspace
\[
\B\{\mathbf x \in \RR^n\colon \sum_{i \in L} x_i = 1\B\}\subset\RR^n
\]
intersecting $W$ transversely. Thus $Q_G$ is contained in a $(n-r-1)$-dimensional subspace of $\RR^n$, and $Q_G$ is $(n-r-1)$-dimensional.
\end{proof}

\textbf{Polytopes.} We refer to~\cite{ziegler2007} for background on polytopes in general. In what follows, let
\[
\ell\colon(x_1, \dots, x_n)\mapsto \sum_{i=1}^nc_ix_i
\]
denote a linear form. Recall that a \textbf{face} $F$ of a polytope $P\subset\RR^n$ is a subset of the form
\[
F = P\cap \{\mathbf x\colon \ell(\mathbf x) = c\}
\]
for some $c \in \RR$ such that (affine) hyperplane $\{\ell(\mathbf x) = c\}$ is a \textbf{supporting hyperplane (for $F$)}, i.e.\ such that
\[
P\subset\{\mathbf x\colon\ell(\mathbf x)\geq c\}
\]
holds. A \textbf{facet} of a polytope is a face of codimension 1.

We will later use the following lemma.

\begin{lem}
\label{lem:dual-trick}
Let $F$ be a face of a polytope $P$ of codimension $d$. Then $F$ is the intersection of some $d$ facets of $P$.
\end{lem}

\begin{proof}
First recall that every face $F$ of a polytope is the intersection of the facets containing it (see~\cite[Thm 3.1.7]{grunbaum2003} or~\cite[Thm 2.7]{ziegler2007}).

Let $G$ be a face of $P$ of codimension $d-1$ with $G\supseteq F$. By induction, we may find $d-1$ facets $G_1, \dots, G_{d-1}$ whose intersection is $G$. It suffices to find a facet $G_*\supseteq F$ not containing $G$, as $F = G \cap G_*$ for any such facet $G_*$. Such a facet $G_*$ must exist; otherwise, the intersection of all facets containing $F$ would contain $G$.
\end{proof}
\section{Faces of $\tilde Q_G$}
\label{sec:faces-of-tilde-QG}
This section contains the main results of the paper: Theorem~\ref{thm:tilde-faces} characterizes faces $\tilde Q_H\subseteq\tilde Q_G$, while Theorem~\ref{thm:non-tilde-faces} characterizes faces $Q_H\subset\tilde Q_G$. The latter theorem requires significantly more work than the former, but technicalities are summarized by Lemma~\ref{lem:ebl}. Both Theorems~\ref{thm:tilde-faces} and~\ref{thm:non-tilde-faces} are proven by analyzing supporting hyperplanes of the relevant subpolytopes (see Lemmas~\ref{lem:tilde-hyperplane-conditions} and~\ref{lem:non-tilde-hyperplane-conditions}), then finding necessary and sufficient combinatorial conditions on $H\subseteq G$ for which a supporting hyperplane exists.

We begin with the following useful lemma:
\begin{lem}
\label{lem:tilde-hyperplane-conditions}
Let $H\subseteq G$ be a subgraph, so $\tilde Q_H\subseteq \tilde Q_G$. The hyperplane
\[
S = \B\{\mathbf x\colon \sum_{i=1}^n c_ix_i = c\B\}
\]
is a supporting hyperplane for $\tilde Q_H$ if and only if:
\begin{enumerate}[(a)]
\item $c = 0$

\item $c_i \geq c_j$ for all $(i,j) \in E(G)$

\item If $(i,j) \in E(G)$, then $c_i = c_j$ if and only if $(i,j) \in E(H)$.
\end{enumerate}
\end{lem}
\begin{proof}
Suppose $S$ is a supporting hyperplane for $\tilde Q_H$, and set
\[
S_\geq \overset{\rm def}=\B\{\mathbf x\colon \sum_{i=1}^n c_ix_i \geq c\B\} 
\]
Since $\mathbf 0 \in \tilde Q_H$ must be in $S$, condition (a) follows. Conditions (b) and (c) respectively follow from the conditions
\begin{equation}
\label{eqn:supp-hyp}
\tilde Q_G\subset S_\geq \qquad \textup{ and } \qquad \tilde Q_H = \tilde Q_G\cap S
\end{equation}
applied to vertices of $\tilde Q_G$. Conversely, if all three conditions (a), (b), and (c) hold, then
\[
\{\mathbf 0, \mathbf e_i - \mathbf e_j\colon (i,j)\in E(G)\}\subset S_\geq \quad\textup{ and } \quad \{\mathbf 0, \mathbf e_i - \mathbf e_j\colon (i,j)\in E(H)\} = \{\mathbf 0, \mathbf e_i - \mathbf e_j\colon (i,j)\in E(G)\} \cap S.
\]
Taking convex hulls, we deduce that~\eqref{eqn:supp-hyp} holds. Thus, $S$ is a supporting hyperplane for $\tilde Q_H$.
\end{proof}

\begin{defn}
Let $H\subseteq G$ be a subgraph, and let $H_1^\un, \dots, H_m^\un$ be the connected components of the underlying undirected graph $H^\un$ of $H$. The directed multigraph $H_\comp$ is the graph with vertex set 
\[
V(H_\comp) = \{H_i^\un\colon i\in[m]\}
\]
and edge multiset 
\[
E(H_\comp) = \{\{(H_i^\un,H_j^\un)\colon \textup{for each edge } (v_i, v_j) \in E(G)\setminus E(H) \textup{ where } v_i \in V(H_i^\un), v_j \in V(H_j^\un)\}\}.\qedhere
\]
\end{defn}

\begin{example}
\label{ex:hcomp-loops-cycles}
The multigraph $H_\comp$ may have multiple edges, self-loops, or directed cycles. For example, let $H\subseteq G$ be as in Figure~\ref{fig:hcomp-loops-cycles} below.
\begin{figure}[ht]
\begin{center}
\includegraphics[scale=0.8]{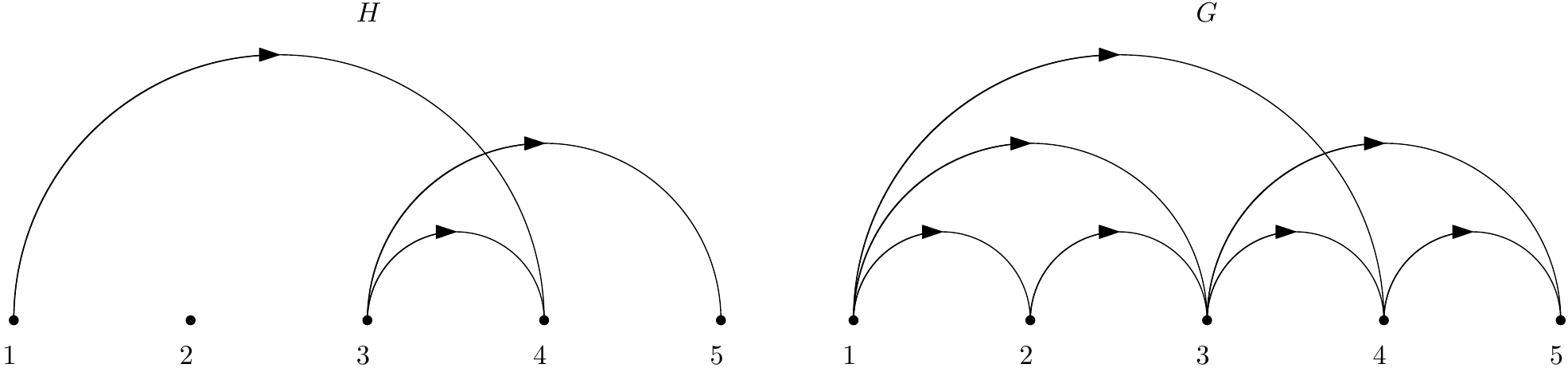}
\end{center}
\caption{The graphs $H$ and $G$ in Example~\ref{ex:hcomp-loops-cycles}.}
\label{fig:hcomp-loops-cycles}
\end{figure}

The graph $H^\un$ has two connected components (with vertex sets $V(H_1^\un) = \{1,3,4,5\}$ and $V(H_2^\un) = \{2\}$), and $H_\comp$ is as in Figure~\ref{fig:hcomp-loops-cycles-2}.
\begin{figure}[ht]
\begin{center}
\includegraphics{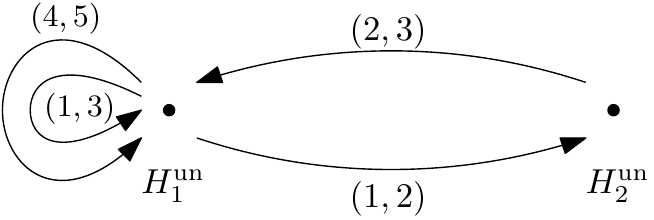}
\end{center}
\label{fig:hcomp-loops-cycles-2}
\caption{The graph $H_\comp$ for $H$ and $G$ in Example~\ref{ex:hcomp-loops-cycles}. The edges $E(H_\comp)$ are labelled by their corresponding edges in $G$.}
\end{figure}

\end{example}

\begin{thm}
\label{thm:tilde-faces}
Let $H\subseteq G$ be a subgraph. The subpolytope $\tilde Q_H\subseteq \tilde Q_G$ is a face of $\tilde Q_G$ if and only if $H_\comp$ is loopless and acyclic.
\end{thm}

\begin{proof}
Suppose $\tilde Q_H$ is a face of $\tilde Q_G$, and take a supporting hyperplane $S = \{\ell(\mathbf x) = c\}$ for $\tilde Q_H$. By condition (c) of Lemma~\ref{lem:tilde-hyperplane-conditions}, the numbers $\{c_i\}_{i \in [n]}$ are constant on connected components of $H$. In particular, if $i$ and $j$ are in the same connected component of $H$, and $(i,j) \in E(G)$, then $(i,j) \in E(H)$; in other words, $H_\comp$ is loopless. By condition (b) and (c) of Lemma~\ref{lem:tilde-hyperplane-conditions}, if $(H_i^\un, H_j^\un) \in E(H_\comp)$, then $c_{v_i} > c_{v_j}$, where $v_i \in V(H_i^\un)$ and $v_j \in V(H_j^\un)$. It follows that $H_\comp$ is acyclic.

Suppose now that $H_\comp$ is loopless and acyclic. We will define numbers $\{c_i\}_{i \in [n]}$ satisfying conditions (b) and (c) of Lemma~\ref{lem:tilde-hyperplane-conditions}, so that
\[
S = \B\{\mathbf x\colon\sum_{i=1}^nc_ix_i = 0\B\}
\]
is a supporting hyperplane for $\tilde Q_H \subseteq\tilde Q_G$. Since $H_\comp$ is loopless and acyclic, we may take a linear extension, i.e.\ a function
\[
f\colon V(H_\comp) \to \{1,\dots, |V(H_\comp)|\}
\]
so that if $(H_i^\un, H_j^\un) \in E(H_\comp)$, then $f(H_i^\un) > f(H_j^\un)$. Each vertex $v_i \in [n]$ is in some connected component $H_i^\un$, the assignment
\[
c_{v_i} = f(H_i^\un)
\]
works.
\end{proof}
We pause to highlight an alternative condition equivalent to looplessness of $H_\comp$.
\begin{prop}
\label{prop:loopless-criterion}
Let $H\subseteq G$ be a subgraph. Then $H_\comp$ is loopless if and only if $H$ is the disjoint union of induced subgraphs $\{G|_{P_i}\}_{P_i \in \mathcal P}$, where $\mathcal P = \{P_i\}$ is a partition of $[n]$.
\end{prop}
\begin{proof}
If $H_\comp$ is loopless, the partition $\mathcal P = \{V(H_i^\un)\}$ works: every edge of $H$ must be contained in some $G|_{V(H_i^\un)}$, so 
\begin{equation}
\label{eqn:loopless-criterion}
H\subseteq \bigsqcup_i G|_{V(H_i^\un)};
\end{equation}
on the other hand, an edge of $G|_{V(H_i^\un)}$ that is not in $H$ becomes a loop in $H_\comp$, so equality holds in~\eqref{eqn:loopless-criterion}. Conversely, suppose $H$ is the disjoint union of induced subgraphs $\{G|_{P_i}\}_{P_i \in \mathcal P}$: if an edge $(i,j) \in E(G)$ connects two vertices $i,j$ in the same connected component of $H^\un$, then $i$ and $j$ are in the same part $P_i \in \mathcal P$, hence must be in $E(H)$. In other words, $H_\comp$ is loopless.
\end{proof}

It remains to characterize faces $Q_H\subset\tilde Q_G$ (Theorem~\ref{thm:non-tilde-faces}). To illustrate the difference between faces $\tilde Q_H\subseteq \tilde Q_G$ and faces $Q_H\subset \tilde Q_G$, consider the following example:

\begin{example}
\label{ex:path-consistency-necessary}
When $H = G = K_3$, the polytope
\[
Q_{K_3} = \conv\{\mathbf e_1 - \mathbf e_2, \mathbf e_1 - \mathbf e_3, \mathbf e_2 - \mathbf e_3\}
\]
is not a face of
\[
\tilde Q_{K_3} = \conv\{\mathbf 0, \mathbf e_1 - \mathbf e_2, \mathbf e_1 - \mathbf e_3, \mathbf e_2 - \mathbf e_3\}.
\]
(It turns out that $Q_{K_3}$ is a triangle and $\tilde Q_{K_3}$ is a rhombus, as Figure~\ref{fig:QK3-example} below shows.) One explanation for this, which turns out to generalize, goes as follows: Suppose that a supporting hyperplane $\{\ell(\mathbf x) = c\}$ for $Q_{K_3}$ exists. Since $\mathbf 0 \not \in Q_{K_3}$, we must have $0 = \ell(\mathbf 0) > c$; up to scaling, we may assume $c = -1$. On one hand, $\ell(\mathbf e_1 - \mathbf e_2) = -1$ and $\ell(\mathbf e_2 - \mathbf e_3) = -1$. On the other hand, $\ell(\mathbf e_1 - \mathbf e_3) = -1$. This is a contradiction.
\begin{figure}[ht]
\begin{center}\includegraphics{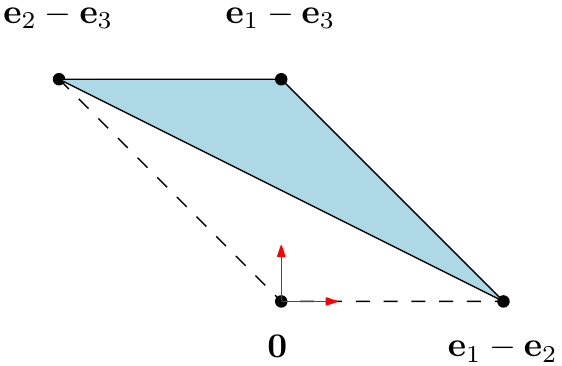}\end{center}
\caption{The root polytopes $Q_{K_3}$ and $\tilde Q_{K_3}$. (The hyperplane $\{x_1 + x_2 + x_3 = 0\}\subset\RR^3$ is identified with $\RR^2$ via the projection $(x_1, x_2, x_3)\mapsto (x_1 - x_2, x_1 - x_3)$; coordinate directions in $\RR^2$ are shown in red.)}
\label{fig:QK3-example}
\end{figure}
\end{example}

\begin{defn}
\label{defn:path-consistent}
A directed acyclic graph $H$ on vertex set $V(H) = [n]$ is \textbf{path consistent} if, for any pair $i,j \in [n]$ and any two undirected paths $\textsf p_{ij}^\un$ and $\textsf q_{ij}^\un$ in $H^\un$ connecting $i$ to $j$, we have
\begin{equation}
\label{eqn:path-consistent}
\#\{(a,b) \in \textsf p_{ij}\colon a < b\} - \#\{(a,b) \in \textsf p_{ij}\colon a > b\} = \#\{(a,b) \in \textsf q_{ij}\colon a < b\} - \#\{(a,b) \in \textsf q_{ij}\colon a > b\}.
\end{equation}
(Here, $\textsf p_{ij}$ and $\textsf q_{ij}$ are the subsets of $E(H)$ whose underlying undirected graph are the paths $\textsf p_{ij}^\un$ and $\textsf q_{ij}^\un$. The sets $\textsf p_{ij}$ and $\textsf q_{ij}$ are not necessarily directed paths.) In other words, the difference between the number of ``correctly'' oriented edges and the number of ``incorrectly'' oriented edges in any path depends only on $i$ and $j$.
\end{defn}

\begin{example}
The complete graph $K_3$ is not path consistent, since the paths $((1,3))$ and $((1,2), (2,3))$ connecting vertices 1 and 3 have one and two correctly oriented edges respectively (cf.\ Example~\ref{ex:path-consistency-necessary}).
\end{example}

\begin{example}
\label{ex:alternating-is-path-consistent}
Any alternating graph $G$ is path consistent. Explicitly, we may apply Lemma~\ref{lem:alternating-is-bipartite} to each connected component of $G^\un$ and obtain a partition $V(G) = [n]$ into two parts $[n] = L\sqcup R$ such that every vertex $i \in L$ is the source of every edge incident to it, and every vertex $j \in R$ is the sink of every edge incident to it. Then, if $\textsf p_{ij}^\un$ is a path connecting $i$ to $j$ in $G^\un$, we have
\begin{equation}
\label{eqn:transitively-closed-Q_H}
\#\{(a,b) \in \textsf p_{ij}\colon a < b\} - \#\{(a,b) \in \textsf p_{ij}\colon a > b\} = \begin{cases} 1 &\textup{ if } i \in L, j \in R\\ 0&\textup{ if } i,j \in L\\ 0&\textup{ if } i,j\in R\\ -1 &\textup{ if } i \in R, j \in L\end{cases}
\end{equation}
so Equation~\eqref{eqn:path-consistent} is satisfied.
\end{example}

While path consistency turns out to be a necessary condition, it is not sufficient, as the next example shows. (The necessity will be the easier half of Theorem~\ref{thm:non-tilde-faces}.)

\begin{example}
\label{ex:admissible-necessary}
Let $H\subset G$ be as in Figure~\ref{fig:2K2-C4}.
\begin{figure}[ht]
\begin{center}
\includegraphics[scale=0.7]{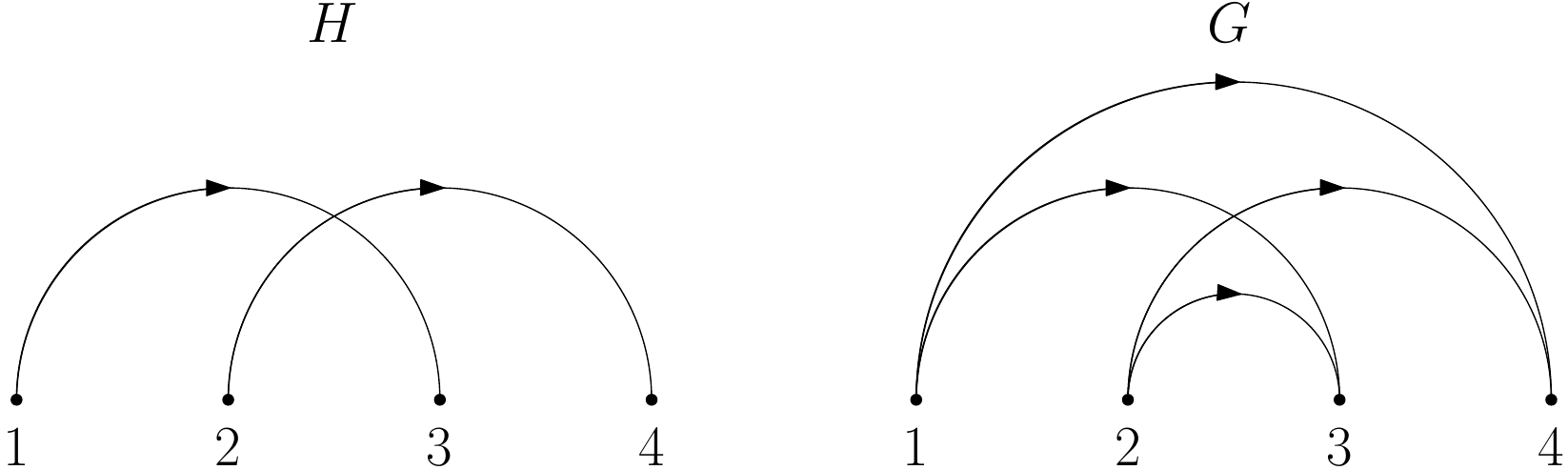}
\end{center}
\caption{The graphs $H$ and $G$ in Example~\ref{ex:admissible-necessary}.}
\label{fig:2K2-C4}
\end{figure}

The root polytope $Q_G$ is a square with affine hull
\[
\{(x_1, x_2, x_3, x_4)\colon x_1 + x_2 = 1, x_3 + x_4 = -1\}\subset \RR^4,
\]
so $\tilde Q_G$ is a square pyramid with apex $\mathbf 0$ (see Figure~\ref{fig:Q2C2-QK4-example}). The subpolytope $Q_H= \conv\{\mathbf e_1 - \mathbf e_3, \mathbf e_2 - \mathbf e_4\}$ is a diagonal of the square face $Q_G$ of $\tilde Q_G$; hence $Q_H$ is not a face of $\tilde Q_G$.

\begin{figure}[ht]
\begin{center}\includegraphics{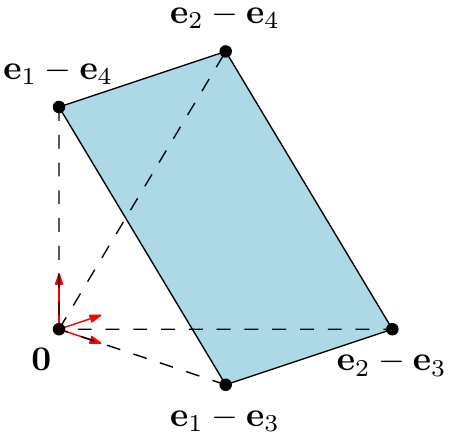}\end{center}
\caption{The root polytopes $Q_H$ and $\tilde Q_G$ in Example~\ref{ex:admissible-necessary}. (The hyperplane $\{x_1 + x_2 + x_3 + x_4 = 0\}\subset\RR^4$ is identified with $\RR^3$ via the projection $(x_1, x_2, x_3, x_4)\mapsto (x_1 - x_2, x_1 - x_3, x_1 - x_4)$; coordinate directions are shown in red.)}
\label{fig:Q2C2-QK4-example}
\end{figure}

Let us explain why $Q_H$ is not a face of $\tilde Q_G$ in a way that will generalize. Suppose that a supporting hyperplane $\{\ell(\mathbf x) = c\}$ for $Q_H$ exists. Since $\mathbf 0 \not \in Q_H$, we must have $0 = \ell(\mathbf 0) > c$; up to scaling, we may assume $c = -1$. Writing
\[
\ell(\mathbf x) = \sum_{i=1}^n c_ix_i,
\]
we have the four conditions
\begin{align*}
(1,3) \in E(H) &\implies c_1 = c_3 - 1,\\
(2,3) \in E(G)\setminus E(H) &\implies c_2 > c_3 - 1,\\
(2,4) \in E(H) &\implies c_2 = c_4 - 1,\\
(1,4) \in E(G)\setminus E(H) &\implies c_1 > c_4 - 1
\end{align*}
on the $c_i$: the first two say $c_2 > c_1$, whereas the last two say $c_1 > c_2$.
\end{example}

We want to introduce a key notion used to generalize Example~\ref{ex:admissible-necessary}. We begin with:
\begin{defn}
Let $H$ be a path consistent graph and assume $H^\un$ is path connected. For any two vertices $u,v \in V(H)$, pick any undirected path $\textsf p^\un$ connecting $u$ to $v$ and set
\[
\ell_{uv} \overset{\rm def}= \#\{(a,b)\in \textsf p\colon a < b\} - \#\{(a,b)\in \textsf p\colon a > b\}.
\]
(This quantity is well-defined because $H$ is path consistent.) We call $u_*\in V(H)$ a \textbf{weight source} if there is a vertex $v_* \in V(H)$ so that 
\[
\ell_{u_*v_*} = \max_{u,v}\ell_{uv}.
\]
Note that a weight source always exists, but is not necessarily unique.
\end{defn}

Although Definition~\ref{defn:weight-function} requires a choice of a weight source $u_*$, we will show in Proposition~\ref{prop:weight-function} that this choice does not matter.

\begin{defn}
\label{defn:weight-function}
Let $H$ be a path consistent graph and assume that $H^\un$ is path connected. Let $u_*$ be a weight source. The \textbf{weight function} (with respect to $u_*$) of $H$ is the function $w_{u_*}\colon V(H)\to \ZZ$ given by
\[
w_{u_*}(i) \overset{\rm def}= \ell_{u_*i}.
\]
\end{defn}
\begin{prop}
\label{prop:weight-function}
Let $H$ be a path consistent graph so that $H^\un$ is path connected. Let $w_{u_*}$ denote the weight function of $H^\un$ with respect to $u_*$. Then:
\begin{enumerate}
\item $w_{u_*}(i) + 1 = w_{u_*}(j)$ for every edge $(i,j) \in E(G)$.

\item $w_{u_*}(i) \geq 0$ for all $i \in V(H)$, and equality holds if and only if $i$ is a weight source.

\item If $u_*'$ is another weight source, then $w_{u_*} = w_{u_*'}$. Thus the weight function of $H$ is well-defined, independent of weight source.
\end{enumerate}
\end{prop}
\begin{defn}
Let $H$ be a path consistent graph (with $H^\un$ possibly disconnected). The \textbf{weight function} of $H$ is the function $w\colon V(H)\to \ZZ$ obtained by gluing together weight functions $w_j\colon V(H_j^\un)\to \ZZ$.
\end{defn}
\begin{proof}[Proof of Proposition~\ref{prop:weight-function}]
Item (1) is a consequence of the fact that concatenating $(i,j)\in E(G)$ to any path connecting $u_*$ to $i$ gives a path connecting $u_*$ to $j$.

More generally, concatenation of paths gives the equality 
\[
\ell_{uv} + \ell_{vw} = \ell_{uw}.
\]
Let $u_*$ be a weight source, and let $v_* \in V(H)$ satisfy $\ell_{u_*v_*} = \max_{u,v}\ell_{uv}$. The equality
\[
\ell_{u_*i} = \ell_{u_*v_*} - \ell_{iv_*}
\]
and the maximality of $\ell_{u_*v_*}$ guarantee that $\ell_{u_*i} \geq 0$. Furthermore equality holds if and only if $\ell_{iv_*} = \ell_{u_*v_*}$, which in turn holds if and only if $i$ is a weight source. This proves item (2).

Finally, let $u_*'$ be another weight source. By part (2), $w_{u_*}(u_*') = 0$ and hence
\[
\ell_{u_*i} = \underbrace{\ell_{u_*u_*'}}_{=0} + \ell_{u_*'i},
\]
so that $w_{u_*} = w_{u_*'}$. This proves item (3).
\end{proof}
\begin{defn}
Let $H\subseteq G$ be a subgraph, and assume $H$ is path consistent. Let $w$ be the weight function of $H$. Each edge $e = (H_i^\un, H_j^\un) \in E(H_\comp)$ of the multigraph $H_\comp$ corresponds to a unique edge $(v_i, v_j) \in E(G)\setminus E(H)$. We define the \textbf{weight decrease} of $e$ to be the quantity
\[
\textsf{wd}(e) \overset{\rm def}= w(v_i) - w(v_j).\qedhere
\]
\end{defn}
While path consistency will be analogous to looplessness of $H_\comp$, the following condition will be analogous to acyclicity of $H_\comp$.
\begin{defn}
\label{defn:admissible}
A subgraph $H\subseteq G$ is \textbf{admissible (with respect to $G$)} if, for every directed cycle $\mathcal C$ in $H_\comp$, the condition
\begin{equation}
\label{eqn:admissible}
\sum_{e \in \mathcal C}\textsf{wd}(e) > -|\mathcal C|
\end{equation}
holds, where the sum in~\eqref{eqn:admissible} runs over the edges $e$ forming the directed cycle $\mathcal C$.
\end{defn}

\begin{example}[cf.\ Example~\ref{ex:admissible-necessary}]
Returning to Example~\ref{ex:admissible-necessary}, we let $H\subset G$ be as in Figure~\ref{fig:2K2-C4}. The graph $H$ is path consistent; the graph $H^\un$ has two connected components $H_1^\un$ and $H_2^\un$ consisting of vertices $\{1,3\}$ and $\{2,4\}$ respectively.

The weight function $w\colon V(H)\to \NN$ sends $w(1) = w(2) = 1$ and $w(3) = w(4) = 2$. The graph $H_\comp$ consists of a single cycle of length 2: there is an edge $e = (H_1^\un,H_2^\un) \in E(H_\comp)$ corresponding to the edge $(1,4) \in E(G)\setminus E(H)$ and its weight decrease is $\textsf{wd}(e) = -1$; there is also an edge $e' = (H_2^\un, H_1^\un) \in E(H_\comp)$ corresponding to the edge $(2,3) \in E(G)\setminus E(H)$ and its weight decrease is $\textsf{wd}(e') = -1$.

The graph $H_\comp$ has a single directed cycle $\mathcal C = \{e,e'\}$, and the condition~\eqref{eqn:admissible}
\[
\textsf{wd}(e) + \textsf{wd}(e') > -2
\]
fails to hold. Thus $H\subseteq G$ is not admissible.
\end{example}

We now have enough language to state our characterization of subgraphs $H\subseteq G$ for which $Q_H\subset\tilde Q_G$ is a face:

\begin{thm}
\label{thm:non-tilde-faces}
Let $H\subseteq G$ be a subgraph of $G$. The subpolytope $Q_H\subset\tilde Q_G$ is a face of $\tilde Q_G$ if and only if $H$ is path consistent and admissible.
\end{thm}

To prove Theorem~\ref{thm:non-tilde-faces}, we will use the following technical lemma; Section~\ref{sec:ebl} is dedicated to its proof, which we feel is unenlightening in the context of this paper.

\begin{lem}
\label{lem:ebl}
Let $H\subseteq G$ be an admissible subgraph of $G$. There is a vector $\mathbf d = (d_v)_{v \in V(H_\comp)} \in \RR^{V(H_\comp)}$ so that
\[
\textup{\textsf{wd}}(e) + d_{s(e)} - d_{t(e)} > -1
\]
for every edge $e \in E(H_\comp)$.
\end{lem}

(Here, $s(e)$ denotes the source of the edge $e$, and $t(e)$ denotes the target of the edge $e$.)

We now prove the following analogue of Lemma~\ref{lem:tilde-hyperplane-conditions} for faces $Q_H\subset\tilde Q_G$:

\begin{lem}
\label{lem:non-tilde-hyperplane-conditions}
Let $H\subseteq G$ be a subgraph, so $Q_H\subset\tilde Q_G$. The hyperplane
\[
S = \B\{\mathbf x\colon\sum_{i=1}^nc_ix_i = c\B\}
\]
is a supporting hyperplane for $Q_H$ if and only if:
\begin{enumerate}[(a)]
\item $c < 0$

\item $c_i \geq c_j + c$ for all $(i,j) \in E(G)$

\item If $(i,j) \in E(G)$, then $c_i = c_j + c$ if and only if $(i,j) \in E(H)$.
\end{enumerate}
\end{lem}
\begin{proof}
Suppose $S$ is a supporting hyperplane for $Q_H$, and set
\[
S_\geq \overset{\rm def}=\B\{\mathbf x\colon \sum_{i=1}^n c_ix_i \geq c\B\} 
\]
Since $\mathbf 0 \not\in Q_H$ must be in $S_\geq \setminus S$, condition (a) follows. Conditions (b) and (c) respectively follow from the conditions
\begin{equation}
\label{eqn:supp-hyp-2}
\tilde Q_G\subset S_\geq \qquad \textup{ and } \qquad Q_H = \tilde Q_G\cap S
\end{equation}
applied to vertices of $\tilde Q_G$. Conversely, if all three conditions (a), (b), and (c) hold, then
\[
\{\mathbf 0, \mathbf e_i - \mathbf e_j\colon (i,j)\in E(G)\}\subset S_\geq \quad\textup{ and } \quad \{\mathbf e_i - \mathbf e_j\colon (i,j)\in E(H)\} = \{\mathbf 0, \mathbf e_i - \mathbf e_j\colon (i,j)\in E(G)\} \cap S.
\]
Taking convex hulls, we deduce that~\eqref{eqn:supp-hyp-2} holds. Thus, $S$ is a supporting hyperplane for $Q_H$.
\end{proof}
\begin{proof}[Proof of Theorem~\ref{thm:non-tilde-faces}]
Let $Q_H\subset \tilde Q_G$ be a face of $\tilde Q_G$, and take a supporting hyperplane
\[
S = \B\{\mathbf x\colon \sum_{i=1}^n c_ix_i = c\B\}
\]
of $Q_H$.

Applying condition (a) of Lemma~\ref{lem:non-tilde-hyperplane-conditions}, we may assume up to scaling $c = -1$. Then, if $\textsf p_{ij}^\un$ is an undirected path in $H^\un$ connecting $i$ to $j$, and $\textsf p_{ij}$ is the overlying directed subgraph of $H$, we have
\[
\#\{(a,b) \in \textsf p_{ij}\colon a < b\} - \#\{(a,b) \in \textsf p_{ij}\colon a > b\} = c_j - c_i
\]
by repeatedly applying condition (c) of Lemma~\ref{lem:non-tilde-hyperplane-conditions} to the edges $e \in \textsf p_{ij}\subseteq E(H)$. This holds for any such path of $H^\un$, so Equation~\eqref{eqn:path-consistent} is satisfied, and $H$ is path consistent. Importantly, we emphasize that when $i,j \in [n]$ are in the same connected component of $H^\un$, then
\begin{equation}
\label{eqn:coeff-weight-equality}
c_j - c_i = w(j) - w(i),
\end{equation}
where $w$ is the weight function of $H$.

Furthermore, let $\mathcal C$ be a directed cycle of $H_\comp$, consisting of edges $\{e_1^\comp, \dots, e_{|\mathcal C|}^\comp\}$ corresponding to edges $\{e_1, \dots, e_{|\mathcal C|}\}\subseteq E(G)\setminus E(H)$. Denote by $s_i, t_i \in [n] = V(G)$ the source and target of the edge $e_i$ respectively. Since $\mathbf e_{s_i} - \mathbf e_{t_i} \not \in Q_H$, condition (c) of~\ref{lem:non-tilde-hyperplane-conditions} says
\begin{equation}
\label{eqn:coeff-inequality}
c_{s_i} - c_{t_i} > -1
\end{equation}
so
\[
\sum_{i=1}^{|\mathcal C|}\textsf{wd}(e_i^\comp) = \sum_{i=1}^{|\mathcal C|} (w(s_i) - w(t_i)) = \sum_{i=1}^{|\mathcal C|}(w(s_{i+1}) - w(t_i)),
\]
with $s_{|\mathcal C| + 1}\overset{\rm def}= s_1$. Since $\mathcal C$ forms a cycle in $H_\comp$, the target of $e_i^\comp \in E(H_\comp)$ is equal to the source of $e_{i+1}^\comp \in E(H_\comp)$. Thus, the vertices $t_i, s_{i+1} \in V(H)$ are in the same connected component of $H^\un$. Then Equations~\eqref{eqn:coeff-weight-equality} and~\eqref{eqn:coeff-inequality} say
\[
\sum_{i=1}^{|\mathcal C|}(w(s_{i+1}) - w(t_i)) = \sum_{i=1}^{|\mathcal C|} (c_{s_{i+1}} - c_{t_i}) = \sum_{i=1}^{|\mathcal C|}(c_{s_i} - c_{t_i})> -|\mathcal C|.
\]
Thus we have verified Equation~\eqref{eqn:admissible} holds for every cycle $\mathcal C$, and $H$ is admissible.

Suppose now that $H$ is path consistent and admissible. It suffices to provide numbers $c_i$, $i \in [n] = V(H)$, so that conditions (b) and (c) of~\ref{lem:non-tilde-hyperplane-conditions} hold for $c = -1$, i.e.
\begin{equation}
\label{eqn:non-tilde-coeff-condition}
c_i - c_j > -1 \textup{ for } (i,j) \in E(G)\setminus E(H) \qquad \textup{ and } \qquad c_i - c_j = -1 \textup{ for } (i,j) \in E(H).
\end{equation}
By Lemma~\ref{lem:ebl}, there exist numbers $d_i$, $i \in V(H_\comp)$, so that
\[
\textsf{wd}(e) + d_{s(e)} - d_{t(e)} > -1.
\]
Now let $v \in [n] = V(H)$ be a vertex of $H$ and suppose $v \in V(H_{v^\comp}^\un)$ is in the $(v^\comp)$-th connected component of $H^\un$. Then
\[
c_v \overset{\rm def}= w(v) + d_{v^\comp},
\]
where $w$ is the weight function of $H$, satisfies Equation~\eqref{eqn:non-tilde-coeff-condition}: if $e = (i,j) \in E(G)\setminus E(H)$ corresponds to $e^\comp = (i^\comp, j^\comp) \in E(H_\comp)$ then
\[
c_i - c_j = \textsf{wd}(e) + d_{i^\comp} - d_{j^\comp} > -1,
\]
while if $(i,j) \in E(H)$ then (as in Equation~\eqref{eqn:coeff-weight-equality})
\[
c_i - c_j = w(i) - w(j) = -1.\qedhere
\]
\end{proof}
\section{Proof of Lemma~\ref{lem:ebl}}
\label{sec:ebl}
This section contains a proof of Lemma~\ref{lem:ebl}. We feel that Lemma~\ref{lem:ebl-extension} might be of independent interest, although it is largely irrelevant in the context of this paper. \textbf{In this section only}, we temporarily allow $G$ to be a directed multigraph.

In what follows, we will treat signed multisets $\mathcal S$ of edges of $G$ as formal sums
\[
\mathcal S = \sum_{e\in E(G)} m_e(\mathcal S)\cdot e
\]
of edges, where $m_e(\mathcal S)$ is the signed multiplicity of $e$ in $\mathcal S$. We identify the set of formal $\ZZ$-linear combinations of edges of $G$ with $\ZZ^{E(G)}$. 

We treat simple directed cycles $\mathcal C$ and directed paths $\textsf p$ a set of edges, i.e.\ as sums
\[
\mathcal C = \sum_{e\in \mathcal C}e \qquad\textup{ and } \qquad \textsf p = \sum_{e\in\textsf p}e.
\]
\begin{defn}
We let $\mathcal M_G$ denote the abelian group of $\ZZ$-linear combinations of simple directed cycles. The elements of $\mathcal M_G$ are called \textbf{formal cycles}, and $\mathcal M_G$ is a $\ZZ$-submodule of $\ZZ^{E(G)}$.
\end{defn}
\begin{example}
\label{ex:MG-relations}
Simple directed cycles may satisfy relations in $\mathcal M_G$. For example, consider $G$ as in Figure~\ref{fig:MG-relations} below.
\begin{figure}[ht]
\begin{center}\includegraphics{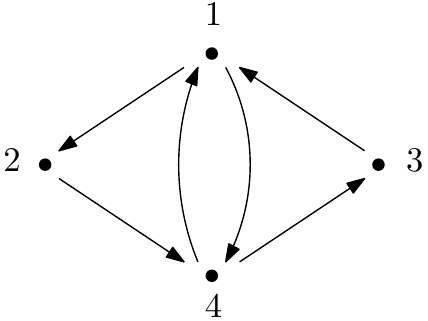}\end{center}
\caption{The graph $G$ in Example~\ref{ex:MG-relations}.}
\label{fig:MG-relations}
\end{figure}

Let $\mathcal C_1 = (1,2) + (2,4) + (4,1)$ and $\mathcal C_2 = (1,4) + (4,3) + (3,1)$. Also let $\mathcal C_3 = (1,2) + (2,4) + (4,3) + (3,1)$ and $\mathcal C_4 = (1,4) + (4,1)$. These are all simple directed cycles, and in $\mathcal M_G$ the relation $\mathcal C_1 + \mathcal C_2 = \mathcal C_3 + \mathcal C_4$ holds.
\end{example}
\begin{defn}
For a formal sum of edges $\mathcal S \in \ZZ^{E(G)}$ and an edge $e \in E(G)$, we let $m_e(\mathcal S)$ denote the coefficient of $e$ in $\mathcal S$. The \textbf{support} of $\mathcal S$ is the set $\{e\in E(G)\colon m_e(\mathcal S)\neq 0\}$ and is denoted $\textup{supp}(\mathcal S)$. We also set
\[
|\mathcal S|\overset{\rm def}=\sum_{e\in E(G)}|m_e(\mathcal S)|.\qedhere
\]
\end{defn}
\begin{lem}
\label{lem:ebl-elements}
For any directed multigraph $G$, the abelian group $\mathcal M_G$ is equal to the set of formal sums $\mathcal S \in \ZZ^{E(G)}$ satisfying
\begin{equation}
\label{eqn:ebl-elements-condition}
\sum_{\substack{e\in E(G)\\s(e)=v}}m_e(\mathcal S)\cdot e = \sum_{\substack{e\in E(G)\\t(e)=v}}m_e(\mathcal S)\cdot e\qquad \textup{ for all }v \in V(G)
\end{equation}
and
\begin{equation}
\label{eqn:ebl-elements-condition-supp}
\textup{supp}(\mathcal S)\subseteq\bigcup_{\mathcal C}\textup{supp}(\mathcal C),
\end{equation}
where the union runs over all simple directed cycles $\mathcal C$ of $G$.
\end{lem}
(As in Section~\ref{sec:faces-of-tilde-QG}, the notation $s(e)$ and $t(e)$ stands for the source and target of the edge $e$ respectively.)
\begin{rem}
\label{rem:ebl-elements}
We will use Lemma~\ref{lem:ebl-elements} in the following way: If $H\subseteq G$ is a subgraph which is obtained as a union of directed cycles, and the formal cycle $\mathcal C \in \mathcal M_G$ has support in $H$, then $\mathcal C \in \mathcal M_H$. (That is to say, although $\mathcal C$ comes as a $\ZZ$-linear combination of directed cycles of $G$, it may be replaced by a $\ZZ$-linear combination of directed cycles of $H$.)

For example, with notation as in Example~\ref{ex:MG-relations}, the formal cycle $\mathcal C_* \overset{\rm def}= \mathcal C_1 + \mathcal C_2 - \mathcal C_3 \in \mathcal M_G$ has support in the subgraph $H\subseteq G$ where $E(H) = \{(1,4), (4,1)\}$. As expected, $\mathcal C_*$ is a $\ZZ$-linear combination of directed cycles of $H$, since $\mathcal C_* = \mathcal C_4$.
\end{rem}
\begin{proof}[Proof of Lemma~\ref{lem:ebl-elements}]
Any simple directed cycle satisfies conditions~\eqref{eqn:ebl-elements-condition} and~\eqref{eqn:ebl-elements-condition-supp}; it follows that any formal cycle satisfies conditions~\eqref{eqn:ebl-elements-condition} and~\eqref{eqn:ebl-elements-condition-supp} as well.

Conversely, suppose $\mathcal S$ is a formal sum
\[
\mathcal S = \sum_{e\in E(G)} m_e(\mathcal S)\cdot e
\]
satisfying~\eqref{eqn:ebl-elements-condition} and~\eqref{eqn:ebl-elements-condition-supp}; our goal is to show that $\mathcal S$ is a formal cycle. Adding directed cycles to $\mathcal S$ if necessary, we may assume $m_e(\mathcal S)\geq 0$ for every $e \in E(G)$.

Thus, it suffices to show that nonnegative formal sums of edges satisfying~\eqref{eqn:ebl-elements-condition} and~\eqref{eqn:ebl-elements-condition-supp} are formal cycles. The remainder of the proof is by induction on $|\mathcal S|$. Specifically, we argue that there exists a simple directed cycle $\mathcal C$ of $G$ whose support is contained in $\textup{supp}(\mathcal S)$; because $\mathcal S - \mathcal C$ is again a nonnegative formal sum of edges satisfying~\eqref{eqn:ebl-elements-condition} and~\eqref{eqn:ebl-elements-condition-supp}, the inductive hypothesis guarantees that $\mathcal S - \mathcal C$ is a formal cycle.

Indeed, pick any edge $e = (s,t) \in \textup{supp}(\mathcal S)$; since conditions~\eqref{eqn:ebl-elements-condition} and~\eqref{eqn:ebl-elements-condition-supp} holds for the vertex $t \in V(G)$, there is another edge $e' \in \textup{supp}(\mathcal S)$ whose source is the vertex $t$. By repeating this process, we obtain edges whose concatenation forms a directed path; this path eventually intersects itself and thus contains a simple directed cycle.
\end{proof}
\begin{lem}
\label{lem:ebl-extension}
Let $G$ be a directed multigraph, and let $c\colon \mathcal M_G \to \RR$ be an additive map such that $c(\mathcal C) > -|\mathcal C|$ for any directed cycle $\mathcal C \in \mathcal M_G$. Then $c$ can be extended to an additive map $c\colon \ZZ^{E(G)}\to \RR$ so that $c(e) > -1$ for all $e \in E(G)$.
\end{lem}

\begin{proof}
The proof is by induction on $|E(G)|$. Note that if $G$ has no directed cycles, the lemma is vacuous, so we may assume $G$ has at least one directed cycle. We enumerate the simple directed cycles of $G$ by $\mathcal C_1, \dots, \mathcal C_r$.

Set
\[
W\overset{\rm def}= \min_{i\in[r]}\B\{\frac{c(\mathcal C_i)}{|\mathcal C_i|}\B\}>-1; \quad I\overset{\rm def}= \B\{i \in [r]\colon \frac{c(\mathcal C_i)}{|\mathcal C_i|} = W\B\}; \quad E_I\overset{\rm def}=\{e\in E(G)\colon e \in \mathcal C_i \textup{ for some } i \in I\}.
\]
Let us define an additive map $c_1\colon \ZZ^{E_I}\to \RR$ by setting $c_1(e) = W$ for all $e \in E_I$.

Treating $E_I\subseteq G$ as a subgraph, observe that Lemma~\ref{lem:ebl-elements} implies any formal cycle $\mathcal C \in \mathcal M_G$ with support in $E_I$ is in fact a formal cycle in $\mathcal M_{E_I}$ (cf.\ Remark~\ref{rem:ebl-elements}). In particular, any simple directed cycle $\{\mathcal C_i\colon i \in I\}$ of $E_I$ satisfies $c(\mathcal C_i) = W|\mathcal C_i|$, so additivity of $c$ implies any formal cycle $\mathcal C$ of $G$ with support in $E_I$ also satisfies $c(\mathcal C) = W|\mathcal C|$.

Again treating $E_I\subseteq G$ as a subgraph, we may form the multigraph $(E_I)_\comp$. We will argue that formal cycles $\mathcal D$ of $(E_I)_\comp$ can be described as follows: its corresponding linear combination of edges $\mathcal D' \subseteq G\setminus E_I$ is the restriction of some formal cycle $\mathcal C$ of $G$ to $E(G)\setminus E_I$, i.e.
\begin{equation}
\label{eqn:ebl-restrict}
\mathcal D' = \sum_{e\in E(G)\setminus E_I} m_e(\mathcal C)\cdot e \qquad\textup{ for some } \mathcal C \in\mathcal M_G.
\end{equation}
First note that $E_I$ is a union of directed cycles and hence its weak components are strongly connected; if $\mathcal D$ is a directed cycle of $(E_I)_\comp$ then the corresponding $\mathcal D' \subseteq G\setminus E_I$ can be completed to a directed cycle $\mathcal C$ of $G$ by appending directed paths in $E_I$. Such a directed cycle $\mathcal C$ satisfies~\eqref{eqn:ebl-restrict}. When $\mathcal D$ is a formal sum of cycles of $(E_I)_\comp$, the corresponding formal sum of cycles of $G$ also satisfies~\eqref{eqn:ebl-restrict}.

We now argue, for any formal cycle $\mathcal D \in \mathcal M_{(E_I)_\comp}$, that the quantity
\[
c_2(\mathcal D)\overset{\rm def}=c(\mathcal C) - W\sum_{e\in E_I}m_e(\mathcal C)
\]
is well-defined, independent of the choice of formal cycle $\mathcal C$ satisfying~\eqref{eqn:ebl-restrict}. Let $\mathcal C_1, \mathcal C_2$ be formal cycles satisfying~\eqref{eqn:ebl-restrict}, and consider the formal sums
\[
\mathcal C_1\cap E_I \overset{\rm def}= \sum_{e\in E_I}m_e(\mathcal C_1)\cdot e \qquad\textup{ and } \qquad\mathcal C_2\cap E_I\overset{\rm def}=\sum_{e\in E_I}m_e(\mathcal C_2)\cdot e.
\]
By definition, we may decompose $\mathcal C_1\cap E_I$ and $\mathcal C_2 \cap E_I$ into a $\ZZ$-linear combination of directed paths of $E_I$, each of which connects endpoints of edges of $\mathcal D'\subseteq G\setminus E_I$. Treating paths as sums of edges, we may write
\[
\mathcal C_1\cap E_I = \sum_i a_i\cdot \textsf p_{i,1} \qquad\textup{ and } \qquad \mathcal C_2\cap E_I = \sum_i b_i\cdot \textsf p_{i,2},
\]
where $a_i, b_i \in \ZZ$. For a directed path $\textsf p$, let $s(\textsf p)$ and $t(\textsf p)$ denote the source and target respectively. For $v \in V(E_I)$ and $j \in \{1,2\}$ let
\[
S(v;j)\overset{\rm def}=\{i\colon s(\textsf p_{i,j}) = v\} \qquad\textup{ and } \qquad T(v;j)\overset{\rm def}=\{i\colon t(\textsf p_{i,j}) = v\}.
\]
Since $\mathcal C_1$ and $\mathcal C_2$ satisfy~\eqref{eqn:ebl-restrict} for the same formal cycle $\mathcal D \in \mathcal M_{(E_I)_\comp}$, we have
\[
\sum_{i\in S(v;1)} a_i = \sum_{i \in S(v;2)} b_i \qquad \textup{ and } \qquad \sum_{i\in T(v;1)} a_i = \sum_{i \in T(v;2)} b_i
\]
for all $v \in V(E_I)$.

Thus, $\mathcal C_1\cap E_I$ and $\mathcal C_2\cap E_I$ can be \emph{simultaneously} completed to a formal cycle of $E_I$, i.e.\ there exist formal cycles $(\mathcal C_1)^I$ and $(\mathcal C_2)^I$ of $E_I$ so that
\[
\begin{cases} m_e((\mathcal C_1)^I) = m_e(\mathcal C_1) \qquad\textup{ for all } e \in E_I\cap \textup{supp}(\mathcal C_1)\\
m_e((\mathcal C_2)^I) = m_e(\mathcal C_2) \qquad\textup{ for all } e\in E_I\cap \textup{supp}(\mathcal C_2)\\
m_e((\mathcal C_1)^I) = m_e((\mathcal C_2)^I) \qquad\textup{ for all other } e \in E_I.
\end{cases}
\]
Note that $\mathcal C_1 + (\mathcal C_2)^I = (\mathcal C_1)^I + \mathcal C_2$ in $\mathcal M_G$, and hence
\[
c(\mathcal C_1) + W|(\mathcal C_2)^I| = c(\mathcal C_2) + W|(\mathcal C_1)^I|.
\]
Rearranging terms, we obtain
\[
c(\mathcal C_1) - W|\mathcal C_1\cap E_I| = c(\mathcal C_2) - W|\mathcal C_2\cap E_I|.
\]
We conclude $c_2\colon \mathcal M_{(E_I)_\comp}\to \RR$ is well-defined.

The function $c_2\colon \mathcal M_{(E_I)_\comp}\to \RR$ is additive, since restriction commutes with summation: if $\mathcal C_1$ and $\mathcal C_2$ are formal cycles of $G$ whose restrictions to $G\setminus E_I$ are $(\mathcal D_1)'$ and $(\mathcal D_2)'$ respectively, then the restriction of $\mathcal C_1 + \mathcal C_2$ to $G\setminus E_I$ is $(\mathcal D_1)'+ (\mathcal D_2)'$.

Furthermore, if $\mathcal D$ is a directed cycle of $(E_I)_\comp$, then minimality of $W$ implies
\[
\frac{c_2(\mathcal D)}{|\mathcal D|} = \frac{c(\mathcal C) - W|\mathcal C\cap E_I|}{|\mathcal D|} > \frac{\frac{c(\mathcal C)}{|\mathcal C|}|\mathcal C| - \frac{c(\mathcal C)}{|\mathcal C|}|\mathcal C\cap E_I|}{|\mathcal D|} = \frac{\frac{c(\mathcal C)}{|\mathcal C|}|\mathcal D|}{|\mathcal D|}= \frac{c(\mathcal C)}{|\mathcal C|} > -1.
\]
Since $|E((E_I)_\comp)| < |E(G)|$, the inductive hypothesis asserts that the function $c_2$ extends to an additive map $c_2\colon \ZZ^{E((E_I)_\comp)}\to \RR$ so that $c_2(e) > -1$ for all $e \in E((E_I)_\comp)$; identifying $E((E_I)_\comp)$ with $E(G)\setminus E_I$ we obtain an additive map $c_2\colon \ZZ^{E(G)\setminus E_I}\to \RR$.

The functions $c_1\colon \ZZ^{E_I}\to \RR$ and $c_2\colon \ZZ^{E(G)\setminus E_I}\to \RR$ glue to a function $\ZZ^{E(G)}\to \RR$ which we claim extends $c\colon \mathcal M_G\to \RR$. To verify this claim, we must check that if $\mathcal C$ is a simple directed cycle of $G$, then
\begin{equation}
\label{eqn:ebl-extension-condition}
\sum_{e\in \mathcal C\cap E_I}c_1(e) + \sum_{e\in \mathcal C\cap (G\setminus E_I)}c_2(e) = c(\mathcal C).
\end{equation}
By the definitions of $c_1$ and $c_2$, we have
\[
\sum_{e\in \mathcal C\cap E_I}c_1(e) = W|\mathcal C\cap E_I| \qquad\textup{ and } \qquad \sum_{e\in \mathcal C\cap (G\setminus E_I)}c_2(e) = c(\mathcal C) - W|\mathcal C\cap E_I|,
\]
so Equation~\eqref{eqn:ebl-extension-condition} is satisfied.
\end{proof}

We can now prove Lemma~\ref{lem:ebl}, restated here for convenience:

\newtheorem*{lem:ebl}{Lemma~\ref{lem:ebl}}
\begin{lem:ebl}
Let $H\subseteq G$ be an admissible subgraph of $G$. There is a vector $\mathbf d = (d_v)_{v \in V(H_\comp)} \in \RR^{V(H_\comp)}$ so that
\begin{equation}
\label{eqn:ebl-required}
\textup{\textsf{wd}}(e) + d_{s(e)} - d_{t(e)} > -1
\end{equation}
for every edge $e \in E(H_\comp)$.
\end{lem:ebl}

\begin{proof}[Proof of Lemma~\ref{lem:ebl}]
Let $M^T$ denote the transpose of the incidence matrix of $H_\comp$, i.e.\ the matrix corresponding to the linear transformation 
\begin{align*}
M^T\colon \RR^{V(H_\comp)} &\to \RR^{E(H_\comp)}\\
\mathbf e_i &\mapsto \sum_{\substack{e\in E(H_\comp)\\s(e) = i}} \mathbf e_e - \sum_{\substack{e\in E(H_\comp)\\ t(e) = i}} \mathbf e_e.
\end{align*}
Let $\textsf{wd}(H_\comp) \in \RR^{E(H_\comp)}$ denote the vector whose component indexed by $e \in E(H_\comp)$ is $\textsf{wd}(e)$. Then Lemma~\ref{lem:ebl} asks for a vector $\mathbf d \in \RR^{V(H_\comp)}$ so that
\begin{equation}
\label{eqn:d-required}
\textsf{wd}(H_\comp) + M^T\mathbf d > -\mathbf 1,
\end{equation}
where $\mathbf 1\in \RR^{E(H_\comp)}$ is the vector whose components are all equal to 1.

The image of $M^T$ is equal to the cut space of $H_\comp$, i.e.\ the space
\[
W = \B\{\mathbf x \in \RR^{E(H_\comp)}\colon \sum_{e \in \mathcal C}x_e = 0\textup{ for all directed cycles $\mathcal C$ of $H_\comp$}\B\};
\]
see e.g.~\cite[Thm.\ II.3.9, Ex.\ II.4.39]{bollobas1998}.

Because $H$ is admissible, the additive function
\begin{align*}
c\colon \mathcal M_{H_\comp}&\to \RR\\
\mathcal C&\mapsto \sum_{e\in\textup{supp}(\mathcal C)}m_e(\mathcal C)\cdot \textsf{wd}(e)
\end{align*}
satisfies $c(\mathcal C) > -|\mathcal C|$ for every directed cycle $\mathcal C$ of $\mathcal M_{H_\comp}$, so Lemma~\ref{lem:ebl-extension} guarantees that $c$ can be extended to an additive function $c\colon \ZZ^{E(G)}\to \RR$ with $c(e) > -1$. Let $\mathbf c \in \RR^{E(H_\comp)}$ denote the vector whose component indexed by $e \in E(H_\comp)$ is $c(e)$; by definition, $\mathbf c > -\mathbf 1$. The condition that
\[
\sum_{e \in \mathcal C} c(e) = \sum_{e \in \mathcal C}\textsf{wd}(e)
\]
for every directed cycle $\mathcal C$ of $H_\comp$ is exactly the condition
\[
\textsf{wd}(H_\comp) - \mathbf c \in W,
\]
so $\textsf{wd}(H_\comp) - \mathbf c = M^T\mathbf v$ for some $\mathbf v \in \RR^{V(H_\comp)}$. Rearranging,
\[
\textsf{wd}(H_\comp) + M^T(-\mathbf v) = \mathbf c > -\mathbf 1,
\]
so $\mathbf d\colonequals -\mathbf v$ satisfies Equation~\eqref{eqn:d-required}.
\end{proof}

\section{Consequences of Theorems~\ref{thm:tilde-faces} and~\ref{thm:non-tilde-faces}; relations to previous results}
In this section, we explore consequences of Theorems~\ref{thm:tilde-faces} and~\ref{thm:non-tilde-faces}. In Corollaries~\ref{cor:portakal1} and~\ref{cor:portakal2} we highlight a result of Portakal characterizing faces of the form $\tilde Q_H\subseteq \tilde Q_G$ for alternating graphs $G$. In Corollary~\ref{cor:transitively-closed-facets}, we show that Theorem~\ref{thm:non-tilde-faces} specializes to a result of Postnikov characterizing facets of the form $Q_H\subset\tilde Q_G$ for transitively closed graphs $G$ (Definition~\ref{defn:transitively-closed}). We also highlight the special case $G = K_n$ in Corollaries~\ref{cor:tilde-faces-K_n} and~\ref{cor:non-tilde-faces-K_n}; the latter corollary corrects a result of Gelfand, Graev, and Postnikov (see Remark~\ref{rem:ggp-comparison}).

We will use the following notation.

\begin{defn}
Let $G$ be a directed graph and let $A\subseteq V(G)$. The set of \textbf{neighbors} of $A$, denoted $N(A)$, is the set
\[
N(A) \overset{\rm def}= \{v\in V(G)\colon (v,a) \in E(G) \textup{ for some } a \in A\}\cup\{v\in V(G)\colon (a,v) \in E(G) \textup{ for some } a \in A\}.
\]
We say $A\subseteq V(G)$ is \textbf{independent} if it is disjoint from $N(A)$.
\end{defn}

Recall (see Lemma~\ref{lem:alternating-is-bipartite}) that the vertex set of an alternating graph may be partitioned into disjoint sets $L$ and $R$ consisting of source and sink vertices respectively. In this setting, Theorem~\ref{thm:tilde-faces} may be recast as follows.

\begin{cor}[{\cite[Thm.\ 3.12]{portakal2019}}]
\label{cor:portakal1}
Let $G$ be an alternating graph and suppose $G^\un$ is connected. The subgraph $H\subseteq G$ defines a facet $\tilde Q_H\subseteq\tilde Q_G$ if and only if $H^\un$ has two connected components and
\begin{equation}
\label{eqn:facet-db}
H = G|_{A\sqcup N(A)}\sqcup G|_{[n]\setminus (A\sqcup N(A))}
\end{equation}
for some set $A\subset R$ of sink vertices.
\end{cor}

\begin{proof}
Suppose first that $\tilde Q_H$ is a facet of $\tilde Q_G$. Since $G$ is connected, we have $\dim(\tilde Q_G) = n-1$ and hence $\dim(\tilde Q_H) = n-2$. Proposition~\ref{prop:root-polytope-dimension-general} implies $H^\un$ must have two connected components which we denote by $H_1^\un$ and $H_2^\un$. Theorem~\ref{thm:tilde-faces} asserts that the two-vertex graph $H_\comp$ is loopless and acyclic; because $G$ is connected the graph $H_\comp$ must have an edge which, without loss of generality, sends $H_1^\comp \in V(H_\comp)$ to $H_2^\comp \in V(H_\comp)$.

Proposition~\ref{prop:loopless-criterion} implies that $H$ is a disjoint union of induced subgraphs of $G$: specifically, we may write
\[
H = G|_{V(H_1^\un)} \sqcup G|_{V(H_2^\un)}.
\]
Set
\[
A\overset{\rm def}= V(H_1^\un) \cap R;
\]
observe that $H_1\subseteq G|_{A\sqcup N(A)}$: every edge of $H_1$ has target in $A$. Furthermore, because $H_1$ is a source vertex in $H_\comp$, every edge $e = (v,a) \in E(G)$ incident to a vertex in $A$ must be in $H_1$. It follows that $H_1 = G|_{A\sqcup N(A)}$. Since $V(H_2^\un) = [n] \setminus V(H_1^\un)$, we conclude that $H$ has the form~\eqref{eqn:facet-db}.

Suppose now that $H^\un$ has two connected components $H_1^\un$ and $H_2^\un$ and has the form~\eqref{eqn:facet-db}. Proposition~\ref{prop:loopless-criterion} implies $H_\comp$ is loopless.

Observe that $(G|_{A\sqcup N(A)})^\un$ and $(G|_{[n]\setminus (A\sqcup N(A))})^\un$ are both connected: if either had (at least) two connected components, then the other would be empty and $H = G$. We conclude that the two vertices of $H_\comp$ correspond to $H_1 \colonequals G|_{A\sqcup N(A)}$ and $H_2\colonequals G|_{[n]\setminus (A\sqcup N(A))}$. Note that no edge of $E(G)\setminus E(H)$ can have target in $A$. Furthermore, since $A\subset R$ we have $N(A) \subseteq L$; hence no edge of $E(G)\setminus E(H)$ can have target in $N(A)$. Hence $H_\comp$ has no edges whose target is $H_1^\comp \in V(H_\comp)$. In total, we have shown $H_\comp$ is loopless and acyclic, so Theorem~\ref{thm:tilde-faces} implies $\tilde Q_H\subseteq \tilde Q_G$ is a face. Since $H$ has two connected components, $\dim(\tilde Q_H) = n-2$ and $\tilde Q_H$ is a facet.
\end{proof}

\begin{cor}[{\cite[Thm.\ 3.17]{portakal2019}}]
\label{cor:portakal2}
Let $G$ be an alternating graph and suppose $G^\un$ is connected. The subgraph $H\subseteq G$ defines a face $\tilde Q_H\subseteq \tilde Q_G$ of codimension $d$ if and only if $H^\un$ has $d+1$ connected components and can be written as the intersection $H = H_1\cap \dots \cap H_d$ of $d$ many graphs for which $\tilde Q_{H_i}$ is a facet of $\tilde Q_G$.
\end{cor}

\begin{proof}
Suppose $\tilde Q_H\subseteq\tilde Q_G$ is a face of codimension $d$. By Lemma~\ref{lem:dual-trick}, it is the intersection of some $d$ facets $F_1, \dots, F_d$ of $\tilde Q_G$. These facets must contain the origin, so $F_i = \tilde Q_{H_i}$ for some graphs $H_i$. Furthermore, since $\tilde Q_H$ has codimension $d$, Proposition~\ref{prop:root-polytope-dimension-general} asserts $H^\un$ must have $d+1$ connected components.

Now suppose $H^\un$ has $d+1$ connected components and assume $H = H_1\cap \dots \cap H_d$, where $\tilde Q_{H_i}$ is a facet of $\tilde Q_G$. Observe that the vertices of the polytope $\tilde Q_H$ are precisely the common vertices of the polytopes $\tilde Q_{H_i}$ for $i \in [d]$. It follows that the polytope $\tilde Q_H$ is the intersection of the polytopes $\tilde Q_{H_i}$ for $i \in [d]$. Hence $\tilde Q_H \subseteq \tilde Q_G$ is a face; because $H^\un$ has $d+1$ connected components, Proposition~\ref{prop:root-polytope-dimension-general} asserts $H^\un$ must have codimension $d$.
\end{proof}

To state Postnikov's result we begin with the following definition:

\begin{defn}
\label{defn:transitively-closed}
A graph $G$ is called \textbf{transitively closed} if whenever $(i,j), (j,k) \in E(G)$ are edges of $G$, then $(i,k) \in E(G)$ is also an edge of $G$.
\end{defn}

\begin{defn}
Let $L, R \subset [n]$ be disjoint subsets of $[n] = V(G)$. The subgraph $G_{L,R}\subseteq G$ is the (alternating) graph whose edge set is
\[
E(G_{L,R}) = \{(i,j) \in E(G)\colon i \in L, j \in R\}.
\]
We call such graphs \textbf{alternating-induced} subgraphs of $G$.
\end{defn}

\begin{example}
\label{ex:K5-LR}
Let $G = K_5$, $L = \{1,3\}$, and $R = \{2,5\}$. Then $G_{L,R}$ is the graph in Figure~\ref{fig:K5-LR}.
\begin{figure}[ht]
\begin{center}
\includegraphics[scale=0.7]{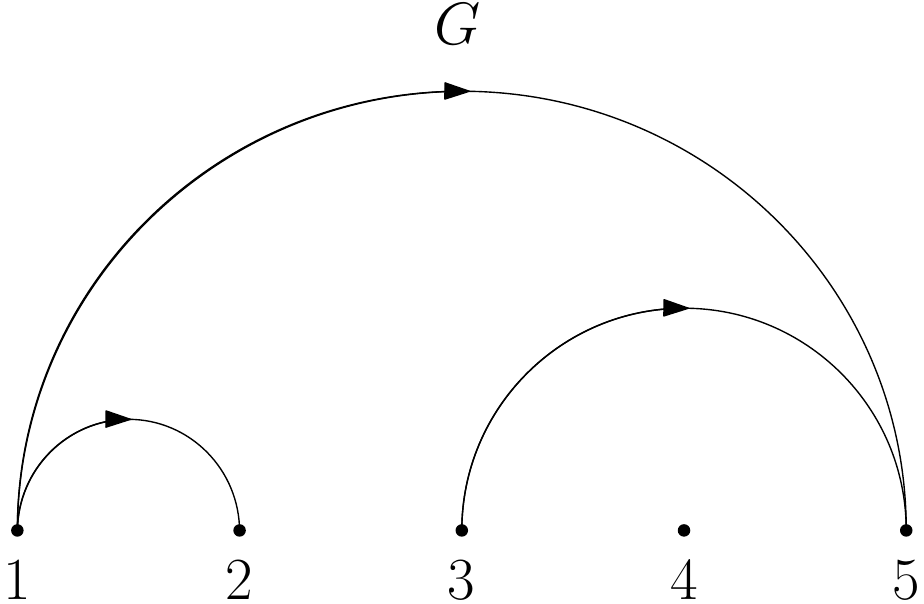}
\end{center}
\caption{The graph $(K_5)_{\{1,3\},\{2,5\}}$ in Example~\ref{ex:K5-LR}.}
\label{fig:K5-LR}
\end{figure}
\end{example}

We will apply the following proposition to obtain corollaries of Theorem~\ref{thm:non-tilde-faces}. Although it is more general than necessary for the purposes of this paper, the proof is essentially the same, so we include it here.

\begin{prop}
\label{prop:trans-closed-non-tilde-necessary}
Let $G$ be transitively closed and let $H\subseteq G$ be a path consistent and admissible subgraph. Then
\begin{equation}
\label{eqn:trans-closed-non-tilde-necessary}
H = \bigsqcup_{P_i \in \mathcal P}(G|_{P_i})_{L_i, R_i}
\end{equation}
is the disjoint union of alternating-induced subgraphs of induced subgraphs of $G$.
\end{prop}

\begin{proof}
Suppose $H\subseteq G$ is path consistent and admissible. We first claim that $H$ is alternating. Indeed, if there exist vertices $i,j,k\in [n] = V(H)$ with $(i,j), (j,k) \in E(H)$, then $(i,k) \in E(G)$ because $G$ is transitively closed. If $(i,k) \in E(H)$ then $H$ is not path consistent, since $\textsf p_{ik} = ((i,j), (j,k))$ and $\textsf q_{ik} = ((i,k))$ violate Equation~\eqref{eqn:path-consistent}, and if $(i,k) \not \in E(H)$ then $H$ is not admissible, since the edge $e = (i,k) \in E(G)\setminus E(H)$ corresponds to a self-loop $e_\comp\in E(H_\comp)$ with $\textsf{wd}(e_\comp) = -2$, violating Equation~\eqref{eqn:admissible}.

Now let $H_i$ denote the overlying directed graph of the connected component $H_i^\un$ of $H^\un$. If $H_i$ is not an isolated vertex, every vertex $v\in V(H_i)$ is either the source of an edge or the target of an edge in $E(H_i)$. Define the (disjoint) subsets
\begin{align*}
L_i &\overset{\rm def}=\{v\in V(H_i)\colon v \textup{ is the source of some $e \in E(H_i)$}\},\\
R_i&\overset{\rm def}=\{v\in V(H_i)\colon v \textup{ is the target of some $e \in E(H_i)$}\}.
\end{align*}
Observe that $H_i \subseteq (G|_{V(H_i)})_{L_i, R_i}$. An edge $e \in E((G|_{V(H_i)})_{L_i, R_i}) \setminus E(H_i)$ corresponds to a loop $e_\comp = (H_i^\un, H_i^\un) \in E(H_\comp)$ with $\textsf{wd}(e_\comp) = -1$, violating Equation~\eqref{eqn:admissible}. It follows that $H_i = (G|_{V(H_i)})_{L_i,R_i}$.
\end{proof}

We use Proposition~\ref{prop:trans-closed-non-tilde-necessary} to deduce Postnikov's result from Theorem~\ref{thm:non-tilde-faces}. For $G = K_n$, this result appeared in the earlier work of~\cite[Prop.\ 13]{cho1999}.

\begin{cor}[{\cite[Prop.\ 13.3]{postnikov2009}}]
\label{cor:transitively-closed-facets}
Let $G$ be transitively closed and suppose $G^\un$ is connected. The subgraph $H\subset G$ defines a facet $Q_H\subseteq\tilde Q_G$ of $\tilde Q_G$ not containing the origin if and only if $H^\un$ is connected and $H = G_{L,R}$ is alternating-induced by some partition $L \sqcup R = [n]$.
\end{cor}

\begin{proof}
Let $Q_H\subset\tilde Q_G$ be a facet. By Theorem~\ref{thm:non-tilde-faces}, $H\subseteq G$ is path consistent and admissible. By Proposition~\ref{prop:trans-closed-non-tilde-necessary}, $H$ has the form~\eqref{eqn:trans-closed-non-tilde-necessary}.

Since $G^\un$ is connected, Proposition~\ref{prop:root-polytope-dimension-general} says $\tilde Q_G$ is $(n-1)$-dimensional, so the facet $Q_H\subseteq \tilde Q_G$ is $(n-2)$-dimensional. Since $H$ is alternating, Proposition~\ref{prop:root-polytope-dimension-general} implies that $H^\un$ has one connected component. It follows that the partition $\mathcal P$ appearing in~\eqref{eqn:trans-closed-non-tilde-necessary} can only contain one part, i.e.\ $H = G_{L,R}$ for some disjoint $L,R\subset[n]$. If $H$ is to contain no isolated vertices, we further obtain $L\sqcup R = [n]$.

Conversely, suppose $H\subseteq G$ is a subgraph so that $H^\un$ is connected and $H = G_{L,R}$ for some $L\sqcup R = [n]$. By Proposition~\ref{prop:root-polytope-dimension-general}, $\dim \tilde Q_G = n-1$ and $\dim Q_H = n-2$, so it suffices to show that $Q_H\subseteq \tilde Q_G$ is a face. Since $H$ is alternating, it is automatically path consistent (as shown in Example~\ref{ex:alternating-is-path-consistent}).

 Note also that $H_\comp$ consists of a single vertex with a self loop corresponding to each edge $e = (i,j) \in E(G)\setminus E(G_{L,R})$, and $\textsf{wd}(e) = 0$ when $i,j \in L$ or $i,j \in R$, whereas $\textsf{wd}(e) = 1$ when $i \in R$ and $j \in L$. In both cases, Equation~\eqref{eqn:admissible} is satisfied and $H$ is admissible. Since $H\subseteq G$ is path consistent and admissible, Theorem~\ref{thm:non-tilde-faces} implies $Q_H\subset\tilde Q_G$ is a face.
\end{proof}

The case $G = K_n$ of Theorems~\ref{thm:tilde-faces} and~\ref{thm:non-tilde-faces} is of special interest, and we spell them out here.

\begin{cor}
\label{cor:tilde-faces-K_n}
The subgraph $H\subseteq K_n$ forms a face $\tilde Q_H\subseteq \tilde Q_{K_n}$ if and only if
\begin{equation}
\label{eqn:tilde-faces-K_n}
H = K_{[1,n_1]} \sqcup K_{[n_1+1,n_2]}\sqcup\dots\sqcup K_{[n_\ell+1,n]}
\end{equation}
is a disjoint union of complete graphs on vertex sets $[n_i + 1, n_{i+1}] \overset{\rm def}=\{n_i+1, n_i+2, \dots, n_{i+1}\}$.
\end{cor}

\begin{proof}
By Theorem~\ref{thm:tilde-faces}, it suffices to characterize subgraphs $H\subseteq K_n$ so that $H_\comp$ is loopless and acyclic. By Proposition~\ref{prop:loopless-criterion}, $H_\comp$ is loopless if and only if $H$ is the disjoint union of induced subgraphs $\{(K_n)|_{P_i}\}_{P_i \in \mathcal P}$, which are just complete graphs $\{K_{P_i}\}_{P_i\in\mathcal P}$ on vertex sets $P_i\subseteq[n]$. The acyclicity of $H_\comp$ implies that if $i < j < k$ and $i \in P_a$, $j \in P_b \neq P_a$, then $k \not \in P_a$. Thus, $P_a$ consists of consecutive numbers $\{n_i + 1, \dots, n_{i+1}\}$. If the partition $\mathcal P = \{P_i\}$ is of the form $P_i = [n_i+1, n_{i+1}]$, it is immediate that $H_\comp$ is acyclic.
\end{proof}

\begin{cor}
\label{cor:non-tilde-faces-K_n}
The subgraph $H\subseteq K_n$ forms a face $Q_H\subset\tilde Q_{K_n}$ if and only if
\begin{equation}
\label{eqn:non-tilde-faces-K_n}
H = (K_{[1,n_1]})_{L_1,R_1} \sqcup (K_{[n_1+1,n_2]})_{L_2,R_2} \sqcup\dots\sqcup (K_{[n_\ell+1,n]})_{L_{\ell+1}, R_{\ell+1}}
\end{equation}
is a disjoint union of alternating-induced subgraphs of complete graphs on vertex sets $[n_i+1,n_{i+1}]$.
\end{cor}

\begin{proof}
By Theorem~\ref{thm:non-tilde-faces}, it suffices to characterize path consistent, admissible subgraphs $H\subseteq K_n$. Let $H\subseteq K_n$ be such a graph. Since $K_n$ is transitively closed, Proposition~\ref{prop:trans-closed-non-tilde-necessary} asserts that
\[
H = \bigsqcup_{P_i \in \mathcal P} (K_{P_i})_{L_i, R_i}
\]
is a disjoint union of alternating-induced subgraphs of complete graphs on vertex sets $P_i\in\mathcal P$. To show that $H$ is of the form~\eqref{eqn:non-tilde-faces-K_n}, it suffices to show that if $i,j,k \in [n] = V(H)$ with $i < j < k$, and $(i,k) \in E(H)$, then either $(i,j) \in E(H)$, $(j,k) \in E(H)$, or $j$ is an isolated vertex. (This would imply that the partition $\mathcal P = \{P_i\}$ can be chosen so that $i,k \in P_*$ implies $j \in P_*$, i.e., so that the parts are consecutive blocks of numbers.)

With the above goal in mind, consider any triple $i < j < k$ with $(i,k)\in E(H)$, and suppose $j$ is not in the same connected component of $H^\un$ as $i$ and $k$; we want to show that $j$ is isolated. If there is an edge $(j,\ell) \in E(H)$, then the edges $e_{i\ell} = (i,\ell)$ and $e_{jk} = (j,k)$ of $E(K_n)$ give rise to a directed cycle $\mathcal C = \{(e_{i\ell})_\comp, (e_{jk})_\comp\}$ in $H_\comp$. Since $\textsf{wd}((e_{i\ell})_\comp) = \textsf{wd}((e_{jk})_\comp) = -1$, Equation~\eqref{eqn:admissible} is violated and $H$ is not admissible.

Similarly, if there is an edge $(\ell,j) \in E(H)$, then the edges $e_{\ell k} = (\ell,k)$ and $e_{ij} = (i,j)$ of $E(K_n)$ give rise to a directed cycle $\mathcal C = \{(e_{\ell k})_\comp, (e_{ij})_\comp\}$ in $H_\comp$. Since $\textsf{wd}((e_{\ell k})_\comp) = \textsf{wd}((e_{ij})_\comp) = -1$, Equation~\eqref{eqn:admissible} is violated and $H$ is not admissible.

Conversely, if $H$ is of the form~\eqref{eqn:non-tilde-faces-K_n}, then it is alternating and hence path consistent, and the directed graph $H_\comp$ is nothing more than the complete graph on $V(H_\comp)$; in particular it is acyclic and Equation~\eqref{eqn:admissible} is satisfied, so $H$ is admissible as well.
\end{proof}

\begin{rem}
\label{rem:ggp-comparison}
The faces $Q_H\subset\tilde Q_{K_n}$ were studied already in~\cite[Prop.\ 8.1]{ggp1997}. Their result contains a mistake; it states that there is a bijection
\[
\rho\colon \{H\colon Q_H\subset\tilde Q_{K_n} \textup{ is a face}\}\longleftrightarrow \{\textup{alternating-induced subgraphs } (K_n)_{L,R}\}
\]
such that $H\subseteq\rho(H)$. This is false for $n = 4$, as for the graphs $H_1$ and $H_2$ in Figure~\ref{fig:H1-H2}, the condition $H\subseteq\rho(H)$ forces $\rho(H_1) = \rho(H_2) = H_2$. Yet, $Q_{H_1}$ is an edge of the triangular facet $Q_{H_2}$ of $\tilde Q_{K_4}$, and indeed Corollary~\ref{cor:non-tilde-faces-K_n} asserts that
\[
H_1 = (K_{\{1,2\}})_{\{1\},\{2\}}\sqcup (K_{\{3,4\}})_{\{3\},\{4\}} \qquad \textup{ and }\qquad H_2 = (K_4)_{\{1,3\},\{2,4\}}
\]
form distinct faces of $Q_H\subset\tilde Q_{K_n}$. 
\begin{figure}[ht]
\begin{center}
\includegraphics[scale=0.7]{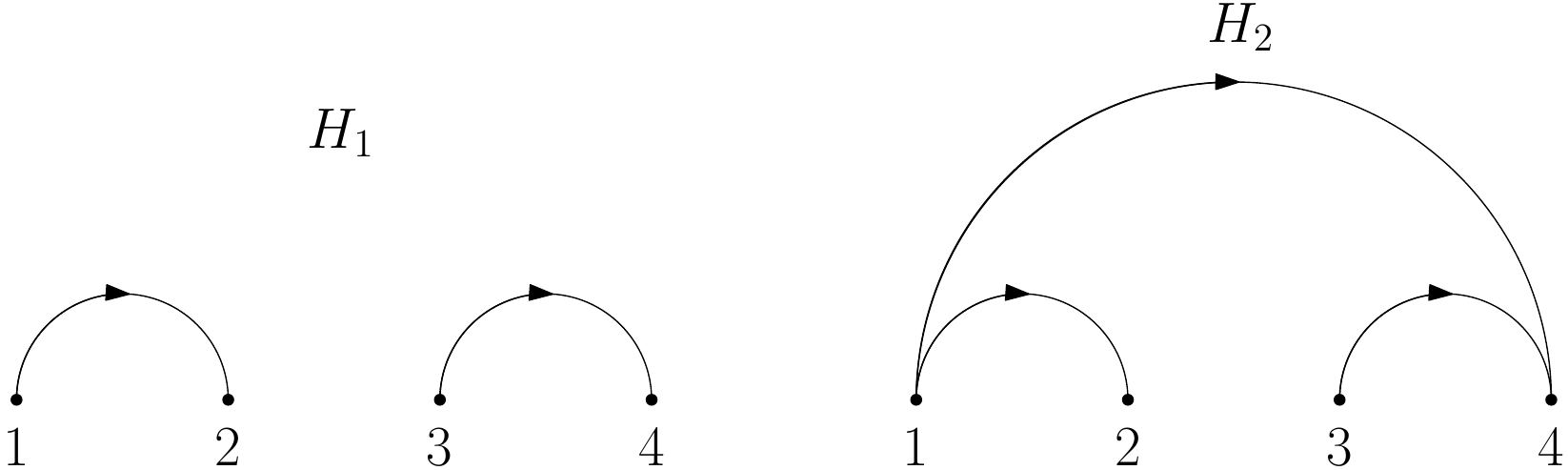}
\end{center}
\caption{The graphs $H_1$ and $H_2$ in Remark~\ref{rem:ggp-comparison}.}
\label{fig:H1-H2}
\end{figure}

Compare~\cite[Prop.\ 8.1]{ggp1997} to Corollary~\ref{cor:non-tilde-faces-K_n}, which asserts that the identity map is a bijection
\[
\textup{id}\colon \{H\colon Q_H\subset\tilde Q_{K_n}\textup{ is a face}\}\longleftrightarrow \{\textup{disjoint unions of alternating-induced subgraphs $(K_{[n_i+1, n_{i+1}]})_{L_i, R_i}$}\}.\qedhere
\]
\end{rem}

\begin{rem}
Corollaries~\ref{cor:tilde-faces-K_n} and~\ref{cor:non-tilde-faces-K_n} give rise to the tantalizing question of explicitly computing the $f$-vector of $\tilde Q_{K_n}$. Specifically, let us highlight that by Proposition~\ref{prop:root-polytope-dimension-general} there are
\begin{align*}
\#\{\textup{graphs of the form~\eqref{eqn:tilde-faces-K_n} with } &n-d \textup{ connected components}\} \\&+\#\{\textup{graphs of the form~\eqref{eqn:non-tilde-faces-K_n} with $n-d-1$ connected components}\}
\end{align*}
faces of dimension $d$. The first term of the summand is easily shown to be
\[
\#\{\textup{graphs of the form~\eqref{eqn:tilde-faces-K_n} with $n-d$ connected components}\} = \binom {n-1}{n-d-1},
\]
as the graph $H$ is uniquely determined by the numbers $1\leq n_1 < \dots < n_{n-d-1} \leq n-1$. 

We record here that a graph $H$ of the form~\eqref{eqn:non-tilde-faces-K_n} arises from a unique choice of $L_i, R_i$ satisfying the additional condition
\begin{equation}
\label{eqn:non-tilde-faces-K_n-bijection}
\min(L_i \cup R_i) \in L_i \qquad \textup{ and } \qquad \max(L_i\cup R_i) \in R_i,
\end{equation}
and that conversely any collection of disjoint sets $L_i, R_i$ satisfying condition~\eqref{eqn:non-tilde-faces-K_n-bijection} and $\max(R_i) < \min(L_{i+1})$ uniquely determines the graph $H$, since we may recover
\[
E(H) = \{(a,b)\colon a \in L_i, b \in R_i \textup{ for some $i$}\}.
\]
In other words, we have a bijection
\[
\{H \textup{ of the form~\eqref{eqn:non-tilde-faces-K_n}}\} \longleftrightarrow \{\textup{disjoint sets $L_i, R_i \subset[n]$ satisfying~\eqref{eqn:non-tilde-faces-K_n-bijection} and $\max(R_i) < \min(L_{i+1})$}\}.
\]
The graph $H$ corresponding to the sets $\{L_1, R_1, \dots, L_\ell, R_\ell\}$ under this bijection is so that $H^\un$ has $\ell$ connected components containing an edge, along with
\[
n - \sum_{i=1}^\ell(|L_i| + |R_i|)
\]
many isolated vertices.
\end{rem}
\section{Acknowledgements}
I am grateful to Karola M\'esz\'aros for her sustained guidance and for many suggestions which improved the quality of this manuscript and its previous drafts. I am also grateful to the anonymous referees for many useful comments. I also thank Florian Frick for helpful discussions about polytopes, as well as Kabir Kapoor, \.Irem Portakal, and especially Seraphina Lee for many stimulating conversations.

\end{document}